\documentclass[a4paper,11pt,twoside]{article}	
\usepackage{amsmath,amssymb,amsthm,amsfonts,color,graphicx}
\usepackage{setspace}
\usepackage[margin=4cm]{geometry}
\newtheorem{theorem}{Theorem}
\DeclareMathOperator{\hr}{\mathbb H^2\times \mathbb R}
\DeclareMathOperator{\rr}{\mathbb R}
\DeclareMathOperator{\tor}{\mathbb T}
\DeclareMathOperator{\hh}{\mathbb H^2}

\newtheorem{proposition}{Proposition}
\newtheorem{remark}{Remark}

\theoremstyle{definition}

\numberwithin{equation}{section}
																																																						
\usepackage{layout,textcomp}
\newtheorem*{Ak}{Acknowledgements}
\title{THE ALEXANDROV PROBLEM IN A QUOTIENT SPACE OF $\mathbb H^2\times \rr$}
\author{Ana Menezes}
\date{}

\begin{document}
\maketitle

\begin{abstract}
We prove an Alexandrov type theorem for a quotient space of $\hr.$ More precisely we classify the compact embedded surfaces with constant mean curvature in the quotient of $\hr$ by a subgroup of isometries generated by a horizontal translation along horocycles of $\mathbb H^2$ and a vertical translation. Moreover, we construct some examples of periodic minimal surfaces in $\hr$ and we prove a multi-valued Rado theorem for small perturbations of the helicoid in $\hr$.

\end{abstract}

\begin{flushleft}
\end{flushleft}

\begin{flushleft}
\end{flushleft}

\section{Introduction}

Alexandrov, in 1962, proved that the only compact embedded constant mean curvature hypersurface in $\rr^n, \mathbb H^n$ and $\mathbb S^n_+$ is the round sphere. Since then, many people have proved an Alexandrov type theorem in other spaces.

For instance, W.T. Hsiang and W.Y. Hsiang \cite{HsiangHsiang} showed that a compact embedded constant mean curvature surface in $\hr$ or in $\mathbb S^2_+\times\rr$ is a rotational sphere. They used the Alexandrov reflection method with vertical planes in order to prove that for any horizontal direction, there is a vertical plane of symmetry of the surface orthogonal to that direction.

To apply the Alexandrov reflection method we need to start with a vertical plane orthogonal to a given direction that does not intersect the surface, and in $\mathbb S^2\times\rr$ this fact is guaranteed by the hypothesis that the surface is contained in the product of a hemisphere with the real line. We remark that in $\mathbb S^2\times\rr,$ we know that there are embedded rotational constant mean curvature tori, but the Alexandrov problem is not completely solved in $\mathbb S^2\times\rr.$ In other simply connected homogeneous spaces with a $4$-dimensional isometry group (Nil$_3, \widetilde{\mbox{PSL}}_2(\rr),$ some Berger spheres), we do not know if the solutions to the Alexandrov problem are spheres. 

In Sol$_3,$ Rosenberg proved that an embedded compact constant mean curvature surface is a sphere \cite{DanielMira}. 

Recently, Mazet, Rodr\' iguez and Rosenberg \cite{Mazet} considered the quotient of $\hr$ by a discrete group of isometries of $\hr$ generated by a horizontal translation along a geodesic of $\mathbb H^2$ and a vertical translation. They classified the compact embedded constant mean curvature surfaces in the quotient space. Moreover, they constructed examples of periodic minimal surfaces in $\hr,$ where by periodic we mean a surface which is invariant by a non-trivial discrete group of isometries of $\hr.$

In this paper we also consider periodic surfaces in $\hr.$ The discrete groups of isometries of $\hr$ we consider are generated by a horizontal translation $\psi$ along horocycles $c(s)$ of $\mathbb H^2$ and/or a vertical translation $T(h)$ for some $h>0.$ In the case the group is the $\mathbb Z^2$ subgroup generated by $\psi$ and $T(h),$ the quotient space $\mathcal M=\hr/[\psi,T(h)]$ is diffeomorphic to $\mathbb T^2\times\rr,$ where $\mathbb T^2$ is the 2-torus. Moreover, $\mathcal M$ is foliated by the family of tori $\tor(s)=c(s)\times\rr/[\psi,T(h)]$ which are intrinsically flat and have constant mean curvature $1/2.$ We prove an Alexandrov type theorem in this quotient space $\mathcal M$.

Moreover, in the last part of this paper, we consider a multi-valued Rado theorem for small perturbations of the helicoid. Rado's theorem (see \cite{Rado}) is one of the fundamental results of minimal surface theory. It is connected to the famous Plateau problem, and states that if $\Omega\subset \rr^2$ is a convex subset and $\Gamma\subset \rr^3$ is a simple closed curve which is graphical over $\partial \Omega,$ then any compact minimal surface $\Sigma\subset\rr^3$ with $\partial \Sigma=\Gamma$ must be a disk which is graphical over $\Omega,$ and then unique, by the maximum principle. In \cite{DeanTinaglia}, Dean and Tinaglia proved a generalization of Rado's theorem. They showed that for a minimal surface of any genus whose boundary is almost graphical in some sense, the minimal surface must be graphical once we move sufficiently far from the boundary. In our work, we consider this problem for minimal surfaces in $\hr$ whose boundary is a small perturbation of the boundary of a helicoid, and we prove that the solution to the Plateau problem is the only compact minimal disk with that boundary (see Theorem $\ref{rado}$). 

This paper is organized as follows. In Section \ref{preliminaries}, we introduce some notation. In Section \ref{sec-alexandrov}, we classify the compact embedded constant mean curvature surfaces in the space $\mathcal M,$ that is, we prove an Alexandrov type theorem for doubly periodic $H$-surfaces (see Theorem $\ref{alexandrov}).$ In Section \ref{sec-examples}, we construct some examples of periodic minimal surfaces in $\hr.$ In Section \ref{sec-rado}, we prove a multi-valued Rado theorem for small perturbations of the helicoid (see Theorem $\ref{rado}$).

{\small
\begin{Ak} This work is part of the author's Ph.D. thesis at IMPA. The author would like to express her sincere gratitude to her advisor Prof. Harold Rosenberg for his constant encouragement and guidance throughout the preparation of this work. 
 The author was financially suported by CNPq-Brazil and IMPA.
\end{Ak}}

\section{Preliminaries}
\label{preliminaries}
Throughout this paper, the Poincar\' e disk model is used for the hyperbolic plane; that is,
$$
\mathbb H^2=\{(x,y)\in\rr^2|\ x^2+y^2<1\}
$$
with  the hyperbolic metric $g_{-1}=\frac{4}{(1-x^2-y^2)^2}g_0,$ where $g_0$ is the Euclidean metric in $\rr^2.$ In this model, the asymptotic boundary $\partial_{\infty}\mathbb H^2$ of $\mathbb H^2$ is identified with the unit circle. Consequently, any point in the closed unit disk is viewed as either a point in $\mathbb H^2$ or a point in $\partial_{\infty}\mathbb H^2.$ We denote by $\textbf{0}$ the origin of $\hh.$

In $\hh$ we consider $\gamma_0, \gamma_1$ the geodesic lines $\{x=0\}, \{y=0\},$ respectively. For $j=0,1,$ we denote by $Y_j$ the Killing vector field whose flow $(\phi_l)_{l\in (-1,1)}$ is given by hyperbolic translation along $\gamma_j$ with $\phi_l(\textbf{0})=$ $(l\sin \pi j, l\cos \pi j)$ and $(\sin \pi j, \cos \pi j)$ as attractive point at infinity. We call $(\phi_l)_{l\in (-1,1)}$ the flow of $Y_j$ even though the family $(\phi_l)_{l\in (-1,1)}$ is not parameterized at the right speed.


We denote by $\pi: \hr\rightarrow\hh$ the vertical projection and we write $t$ for the height coordinate in $\hr.$ In what follows, we will often identify the hyperbolic plane $\hh$ with the horizontal slice $\{t=0\}$ of $\hr.$ The vector fields $Y_j,j=0,1,$ and their flows naturally extend to horizontal vector fields and their flows in $\hr.$

Consider any geodesic $\gamma$ that limits to the point $p_0\in\partial_\infty\hh$ at infinity parametrized by arc length. Let $c(s)$ denote the horocycle in $\hh$ tangent to $\partial_\infty \hh$ at $p_0$ that intersects $\gamma$ at $\gamma(s).$ Given two points $p,q \in c(s),$ we denote by $\psi:\hr\rightarrow\hr$ the parabolic translation along $c(s)$ such that $\psi(p)=q.$


We write $\overline{pq}$ to denote the geodesic arc between the two points $p,q$ of $\hr.$

\section{The Alexandrov problem for doubly periodic constant mean curvature surfaces}
\label{sec-alexandrov}

Take two points $p,q$ in a horocycle $c(s),$ and let $\psi$ be the parabolic translation along $c(s)$ such that $\psi(p)=q.$ We have $\psi(c(s))=c(s)$ for all $s.$ Consider $G$ the $\mathbb Z^2$ subgroup of isometries of $\hr$ generated by $\psi$ and a vertical translation $T(h),$ for some positive $h.$ We denote by $\mathcal M$ the quotient of $\hr$ by $G.$ The manifold $\mathcal M$ is diffeomorphic but not isometric to $\tor^2\times \rr$ and is foliated by the family of tori $\tor(s)=(c(s)\times\rr)/G,$ $s\in \rr,$ which are intrinsically flat and have constant mean curvature $1/2.$ Thus the tori $\tor(s)$ are examples of compact embedded constant mean curvature surfaces in $\mathcal M$.


We have the following answer to the Alexandrov problem in $\mathcal M.$

\begin{theorem}
\label{alexandrov}
Let $\Sigma \subset \mathcal M$ be a compact immersed surface with constant mean curvature $H$. Then $H\geq \frac{1}{2}.$ Moreover:

\begin{enumerate}
\item If $H=\frac{1}{2},$ then $\Sigma$ is a torus $\mathbb T(s),$ for some $s.$

\item If $H>\frac{1}{2}$ and $\Sigma$ is embedded, then $\Sigma$ is either the quotient of a rotational sphere, or the quotient of a vertical unduloid (in particular, a vertical cylinder over a circle).
\end{enumerate}

\end{theorem}

\begin{proof} 
Let $\Sigma$ be a compact immersed surface in $\mathcal M$ with constant mean curvature $H.$ As $\Sigma$ is compact, there exist $s_0\leq s_1\in \rr$ such that $\Sigma$ is between $\tor(s_0)$ and $\tor(s_1),$ and it is tangent to $\tor(s_0),\tor(s_1)$ at points $q,p,$ respectively, as illustrated in Figure \ref{fig1}.

\begin{figure}[h!]
  \centering
 \includegraphics[height=4cm]{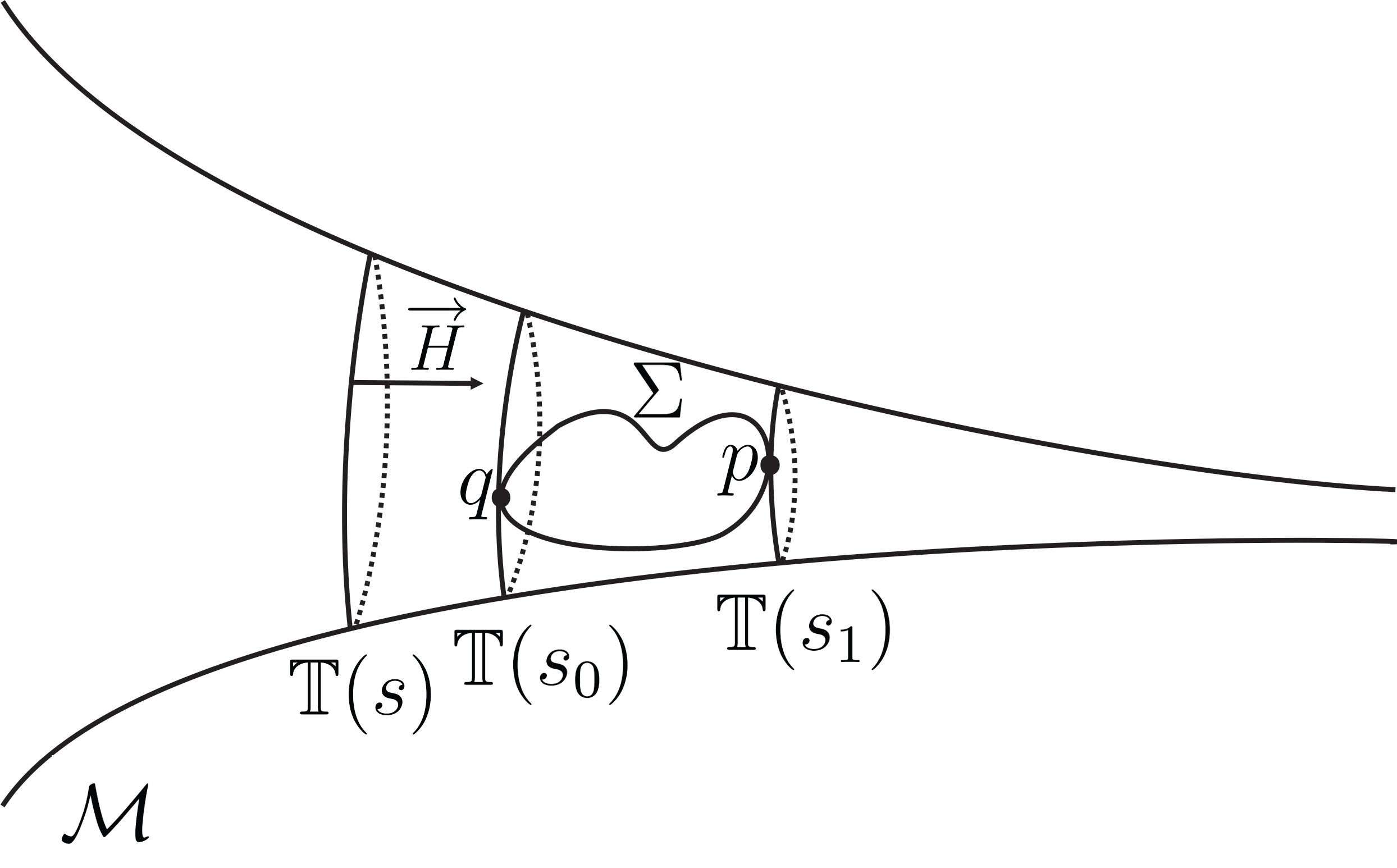}
\caption{$\Sigma\subset \mathcal M.$}
\label{fig1}
\end{figure}

For $s<s_0,$ the torus $\tor(s)$ does not intersect $\Sigma,$ and $\Sigma$ stays in the mean convex region bounded by $\tor(s)$. 
By comparison at $q,$ we conclude that $H\geq \frac{1}{2}.$ If $H=\frac{1}{2},$ then by the maximum principle, $\Sigma$ is the torus $\tor(s_0),$ and we have proved the first part of the theorem.


To prove the last part, suppose $\Sigma$ is embedded and consider the quotient space $\widetilde{\mathcal M}=\mathbb H^2\times\rr/[T(h)],$ which is diffeomorphic to $\mathbb H^2\times\mathbb S^1.$ Take a connected component $\widetilde{\Sigma}$ of the lift of $\Sigma$ to $\widetilde{\mathcal M},$ and denote by $\tilde{c}(s)$ the surface $c(s)\times \mathbb S^1.$ Observe that $\tilde{c}(s)$ is the lift of $\tor(s)$ to $\widetilde{\mathcal M}.$ Moreover, let us consider two points $\tilde{p},\tilde{q}\in\widetilde{\Sigma}$ whose projections in $\mathcal M$ are the points $p,q,$ respectively. 


It is easy to prove that $\widetilde\Sigma$ separates $\widetilde{\mathcal M}.$ In fact, suppose by contradiction this is not true, then we can consider a geodesic arc $\alpha: (-\epsilon, \epsilon)\rightarrow \widetilde{\mathcal M}$ such that $\alpha(0)\in \widetilde\Sigma, \alpha'(0)\in T\widetilde\Sigma^{\bot}$ and we can join the points $\alpha(-\epsilon), \alpha(\epsilon)$ by a curve that does not intersect $\widetilde\Sigma,$ hence we obtain a Jordan curve, which we still call $\alpha,$ whose intersection number with $\widetilde\Sigma$ is 1 modulo 2. 
Notice that the distance between $\widetilde\Sigma$ and $\tilde{c}(s_0)$ is bounded. Since we can homotop $\alpha$ so it is arbitrarily far from $\tilde{c}(s_0),$ we conclude that a translate of $\alpha$ does not intersect $\widetilde\Sigma,$ contradicting the fact that the intersection number of $\alpha$ and $\widetilde\Sigma$ is 1 modulo 2. Thus $\widetilde\Sigma$ does separate $\widetilde{\mathcal M}.$

Let us call $A$ the mean convex component of $\widetilde{\mathcal M}\setminus\widetilde\Sigma$ with boundary $\widetilde\Sigma$ and $B$ the other component. Hence $\widetilde{\mathcal M}\setminus\widetilde\Sigma=A\cup B.$ 




Let $\gamma$ be a geodesic in $\mathbb H^2$ that limits to $p_0\in\partial_\infty\mathbb H^2, \gamma(+\infty)=p_0$ (the point where the horocycles $c(s)$ are centered) and let us assume that $\gamma$ intersects $\widetilde\Sigma$ in \textit{at least} two points. 

Consider $(l_t)_{t\in \rr}$ the family of geodesics in $\mathbb H^2$ orthogonal to $\gamma$ and denote by $P(t)$ the totally geodesic vertical annulus $l_t\times \mathbb S^1$ of $\widetilde{\mathcal M}=\mathbb H^2\times \mathbb S^1$ (see Figure \ref{gamma1}). Since $\widetilde\Sigma$ is a lift of the compact surface $\Sigma,$ it stays in the region between $\tilde{c}(s_0)$ and $\tilde{c}(s_1),$ and the distance from any point of $\widetilde\Sigma$ to $\tilde{c}(s_0)$ and to $\tilde{c}(s_1)$ is uniformly bounded. 

\begin{figure}[h!]
\centering
\includegraphics[height=4.7cm]{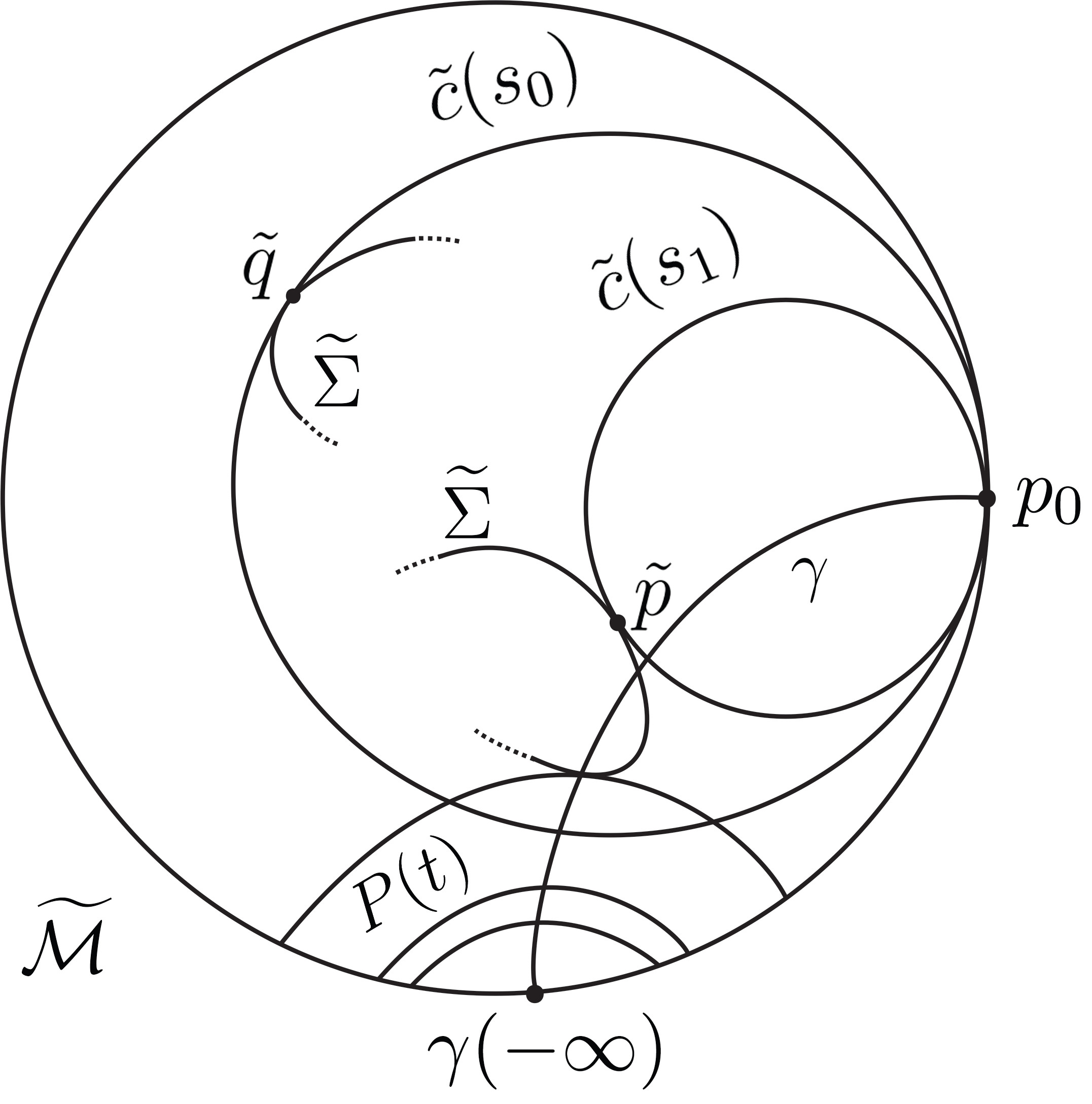}
\caption{The family of totally geodesic annuli $P(t).$}
\label{gamma1}
\end{figure}

By our choice of $\gamma,$ the ends of each $P(t)$ are outside the region bounded by $\tilde{c}(s),$ hence $P(t)\cap\widetilde\Sigma$ is compact for all $t.$ Moreover, for $t$ close to $-\infty$, $P(t)$ is contained in $B$ and $P(t)\cap\widetilde\Sigma$ is empty. Then start with $t$ close to 
$-\infty$ and let $t$ increase until a first contact point between $\widetilde\Sigma$ and some vertical annulus, say $P(t_0).$ In particular, we know that the mean curvature vector of $\widetilde\Sigma$ does not point into $\bigcup_{t\leq t_0}P(t).$


Continuing to increase $t$ and starting the Alexandrov reflection procedure for $\widetilde\Sigma$ and the family of vertical totally geodesic annuli $P(t),$ we get a first contact point between the reflected part of $\widetilde\Sigma$ and $\widetilde\Sigma,$ for some $t_1\in \rr.$ Observe that this first contact point occurs because we are assuming that the geodesic $\gamma$ intersects $\widetilde\Sigma$ in at least two points.

Then $\widetilde\Sigma$ is symmetric with respect to $P(t_1).$ As $\widetilde\Sigma \cap\left(\bigcup_{t_0\leq t\leq t_1} P(t)\right)$ is compact, $\widetilde\Sigma$ is compact. Hence, given any horizontal geodesic $\alpha$ we can apply the Alexandrov procedure with the family of totally geodesic vertical annuli $Q(t)=\tilde{l}_t\times \mathbb S^1,$ where $(\tilde{l}_t)_{t\in\rr}$ is the family of horizontal geodesics orthogonal to $\alpha,$ and we obtain a symmetry plane for $\widetilde\Sigma$. 

Hence we have shown that if some geodesic that limits to $p_0$ intersects $\widetilde\Sigma$ in two or more points, then $\widetilde\Sigma$ lifts to a rotational cylindrically bounded surface $\bar{\Sigma}$ in $\hr.$ If $\bar{\Sigma}$ is not compact then $\bar{\Sigma}$ is a vertical unduloid, and if $\bar{\Sigma}$ is compact we know by Hsiang-Hsiang's theorem \cite{HsiangHsiang} $\bar{\Sigma}$ is a rotational sphere. Therefore, we have proved that in this case $\Sigma\subset\mathcal M$ is either the quotient of a rotational sphere or the quotient of a vertical unduloid.

Now to finish the proof let us assume that every geodesic that limits to $p_0$ intersects $\widetilde\Sigma$ in \textit{at most} one point. In particular, the geodesic $\beta$ that limits to $p_0$ and passes through $\tilde p\in \tilde{c}(s_1)$ intersects $\widetilde\Sigma$ only at $\tilde p.$ Write $\beta^{-}$ to denote the arc of $\beta$
 between $\beta(-\infty)$ and $\tilde p$ (see Figure \ref{beta}).
 
\begin{figure}[h!]
\centering
\includegraphics[height=4.7cm]{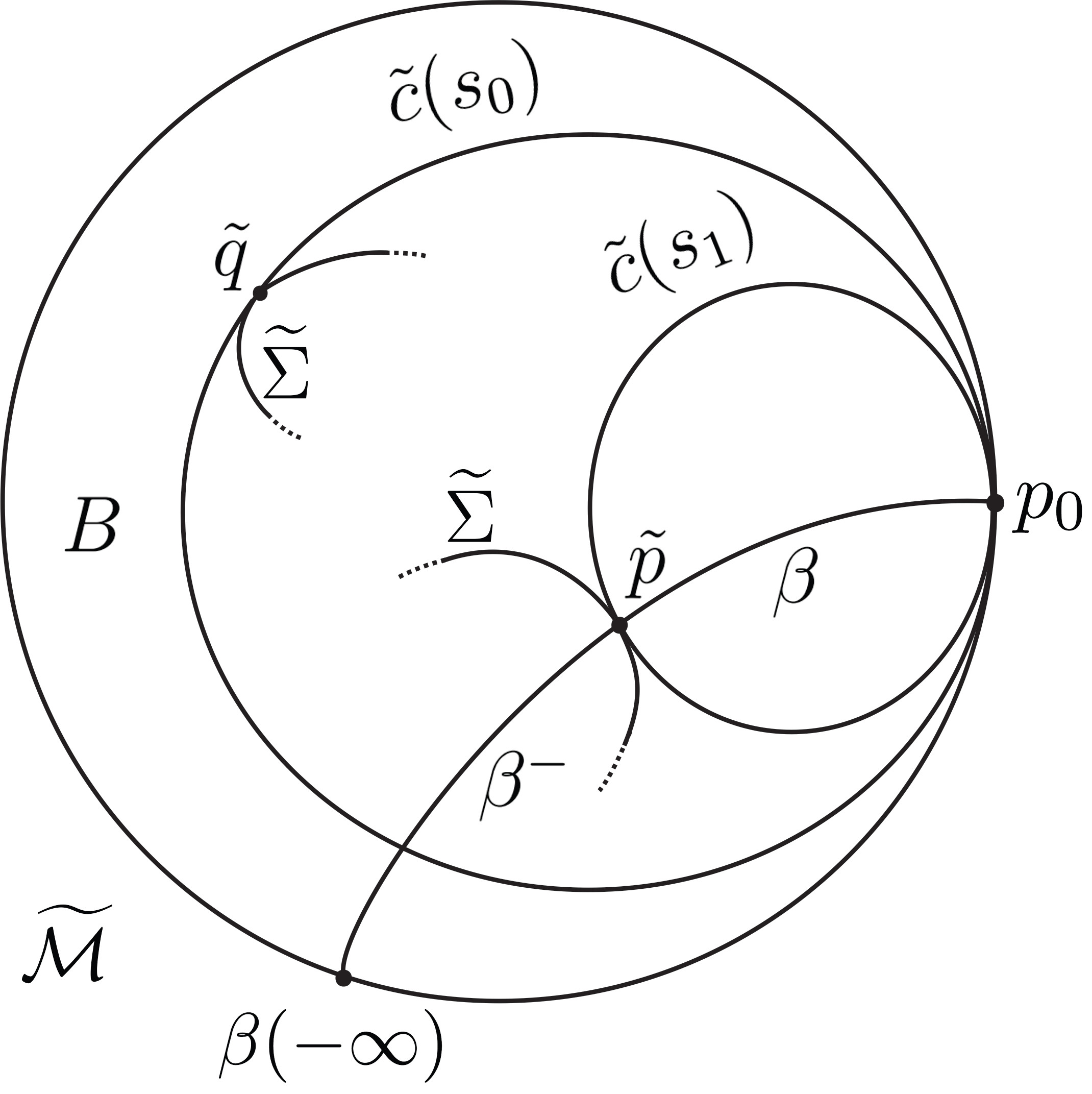}
\caption{Geodesic $\beta.$}
\label{beta}
\end{figure}

As $\beta\cap\widetilde\Sigma=\{\tilde p\},$ we have $\beta^{-}\cap\widetilde\Sigma=\emptyset$ and then $\beta^{-}\subset B,$ since $\widetilde\Sigma$ separates $\widetilde{\mathcal M}.$

Hence at the point $\tilde p\in \widetilde\Sigma\cap\tilde{c}(s_1),$ the mean curvature vectors of $\widetilde\Sigma$ and $\tilde{c}(s_1)$ point to the mean convex side of $\tilde{c}(s_1)$ and $\widetilde\Sigma$ lies on the mean concave side of $\tilde{c}(s_1),$ then by comparison we get $H\leq \frac{1}{2}.$ But we already know that $H\geq \frac{1}{2}.$ Hence $H=\frac{1}{2}$ and $\widetilde\Sigma=\tilde{c}(s_1),$ by the maximum principle. Therefore, in this case we conclude $\Sigma=\tor(s_1).$
\end{proof}

\begin{remark}
Note that a vertical unduloid, contained in a cylinder $D\times \rr$ and invariant by a vertical translation $T(l)$ in $\hr,$ passes to the quotient space $\mathcal M=\hr/[\psi,T(h)]$ as an embedded surface if the quotient of $D$ is embedded and the number $l$ is a multiple of $h$. Analogously, a rotational sphere of height $l$ contained in a cylinder $D\times\rr$ in $\hr$ passes to the quotient as an embedded surface if $l<h$ and the quotient of $D$ is embedded in $\mathcal M.$
\end{remark}

\section{Periodic minimal surfaces}
\label{sec-examples}

In this section we are interested in constructing some new examples of periodic minimal surfaces in $\hr$ invariant by a subgroup of isometries, which is either isomorphic to $\mathbb Z^2,$ or generated by a vertical translation, or generated by a screw motion. In fact, we only consider subgroups generated by a parabolic translation $\psi$ along a horocycle and/or a vertical translation $T(h),$ for some $h>0$. 

Periodic minimal surfaces in $\rr^3$ have received great attention since Riemann, Schwarz, Scherk (and many others) studied them. They also appear in the natural sciences. In \cite{MeeksRosenberg1}, Meeks and Rosenberg proved that a periodic properly embedded minimal surface of finite topology (in $\rr^3/G, G$ a discrete group of isometries acting properly discontinuously on $\rr^3, G\neq (1)$) has finite total curvature and the ends are asymptotic to standard ends (planar, catenoidal, or helicoidal). In a joint paper with Hauswirth \cite{HauswirthMenezes}, we consider the same study for periodic minimal surfaces in $\hr.$ The first step is to understand what are the possible models for the ends in the quotient. This is one reason to construct examples.

\subsection{Doubly periodic minimal surface}
\label{sec-41}

In $\hh$ consider two geodesics $\alpha,\beta$ that limit to the same point at infinity, say $\alpha(-\infty)=p_0=\beta(-\infty).$ Denote $B=\alpha(+\infty)$ and $D=\beta(+\infty).$ Take a geodesic $\gamma$ contained in the region bounded by $\alpha$ and $\beta$ that limits to the same point $p_0$ at infinity. Parametrize these geodesics so that $\alpha(t)\rightarrow B,\beta(t)\rightarrow D$ and $\gamma(t)\rightarrow p_0$ when $t\rightarrow+\infty.$



Fix $h>\pi$ and consider the following Jordan curve:
$$\begin{array}{rcl}
\Gamma_{t}&=& \overline{(\alpha(t),0)\, (\gamma(t),0)}\cup\overline{(\alpha(t),0)\, (\alpha(t),h)}\cup\overline{(\beta(t),0)\, (\gamma(t),0)}\\
&&\\
&&\cup\overline{(\beta(t),0)\, (\beta(t),h)}\cup\overline{(\alpha(t),h)\, (\gamma(t),h)}\cup\overline{(\beta(t),h)\, (\gamma(t),h)}
\end{array}$$
as illustrated in Figure \ref{figure11}.

\begin{figure}[h]
 \centering
\includegraphics[height=5cm]{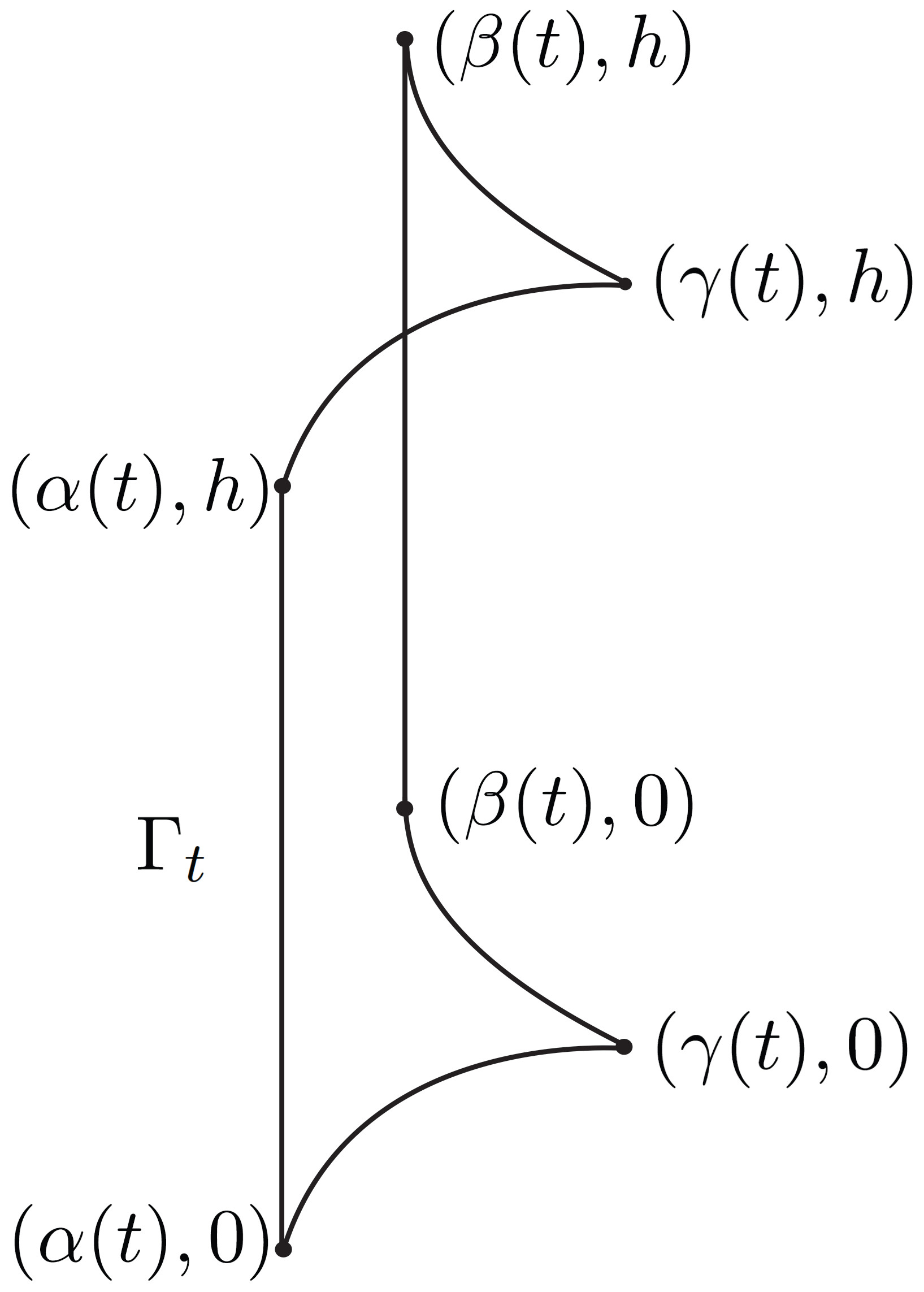}
\caption{Curve $\Gamma_{t}$.}
\label{figure11}
\end{figure}

Consider a least area embedded minimal disk $\Sigma_{t}$ with boundary $\Gamma_{t}.$ Let $Y$ be the Killing field whose flow $(\phi_l)_{l\in\mathbb R}$ is given by translation along the geodesic $\gamma.$  Notice that $\Gamma_{t}$ is transversal to the Killing field $Y.$ Hence given any geodesic $\bar{\gamma}$ orthogonal to $\gamma,$ we can use the Alexandrov reflection technique with the foliation of $\hr$ by the vertical planes $(\phi_l(\bar{\gamma}))_{l\in\mathbb R}$ to show that $\Sigma_{t}$ is a $Y$-Killing graph. In particular, $\Sigma_{t}$ is stable and unique (see \cite{NR-Simply}, Lemma 2.1). This gives uniform curvature estimates for $\Sigma_{t_0}$ for points far from the boundary (see \cite{RosenbergSouam}, Main Theorem). Rotating $\Sigma_t$ by angle $\pi$ around the geodesic arc $\overline{(\alpha(t),0),(\gamma(t),0)}$ gives a minimal surface that extends $\Sigma_t,$ has int$\overline{(\alpha(t),0),(\gamma(t),0)}$ in its interior, and is still a $Y$-Killing graph. Thus we get uniform curvature estimates for $\Sigma_t$ in a neighborhood of $\overline{(\alpha(t),0),(\gamma(t),0)}.$ This is also true for the three other horizontal geodesic arcs in $\Gamma_{t}.$

Observe that for any $t,$ $\Sigma_{t}$ stays in the half-space determined by $\overline{BD}\times\rr$ that contains $\Gamma_{t},$ by the maximum principle.

As $h>\pi,$ we can use as a barrier the minimal surface $S_h\subset \mathbb H^2\times (0,h)$ which is a vertical bigraph with respect to the horizontal slice $\{t=\frac{h}{2}\}.$ The surface $S_h$ is invariant by translations along the horizontal geodesic $\gamma_0=\{x=0\}$ and its asymptotic boundary is $(\tau\times\{0\})\cup \overline{(0,1,0)(0,1,h)}$ $\cup (\tau\times\{h\})\cup \overline{(0,-1,0)(0,-1,h)},$ where $\tau=\partial_{\infty}\mathbb H^2\cap\{x>0\}.$ For more details about the surface $S_h,$ see \cite{Mazet, MazetRodriguezRosenberg, SaEarp}.

For $l$ sufficiently large, the translated surface $\phi_l(S_h)$ does not intersect $\Sigma_{t};$ hence the surface $\Sigma_{t}$ is contained between $\phi_{l}(S_h)$ and $\overline{BD}\times\rr.$ 
 
Notice that when $t\rightarrow +\infty,$ $\Gamma_{t}$ converges to $\Gamma,$ where
 $$\begin{array}{rcl}
 \Gamma&=&(\alpha\times\{0\})\cup(\beta\times\{0\})\cup(\alpha\times\{h\})\cup\\
 &&(\beta\times\{h\}) \cup\overline{(D,0)(D,h)}\cup\overline{(B,0)(B,h)}.
 \end{array}$$
 
Therefore, as we have uniform curvature estimates and barriers at infinity, there exists a subsequence of $\Sigma_{t}$ that converges to a minimal surface $\Sigma,$ where $\Sigma$ lies in the region of $\mathbb H^2\times[0,h]$ bounded by $\alpha\times \rr,$ $\beta\times\rr,$ $\overline{BD}\times\rr$ and $\phi_{l}(S_h);$ with boundary $\partial\Sigma=\Gamma.$ 

Hence the surface obtained by reflection in all horizontal boundary geodesics of $\Sigma$ is invariant by $\psi^2$ and $T(2h),$ where $\psi$ is the horizontal translation along horocycles that sends $\alpha$ to $\beta.$ Moreover, this surface in the quotient space $\hr/[\psi^2, T(2h)]$ is topologically a sphere minus four points. Two ends are asymptotic to vertical planes and two are asymptotic to horizontal planes (cusps), all of them with finite total curvature.

\begin{proposition}
There exists a doubly periodic minimal surface (invariant by horizontal translations along a horocycle and by a vertical translation) such that, in the quotient space, this surface is topologically a sphere minus four points, with two ends asymptotic to vertical planes and two asymptotic to horizontal planes, all of them with finite total curvature.
\end{proposition}

\subsection{Vertically periodic minimal surfaces}

Take $\alpha$ any geodesic in $\mathbb H^2\times \{0\}.$ For $h>\pi,$ consider the vertical segment $\alpha(-\infty)\times [0,2h],$ and a point $p\in \partial_{\infty}\mathbb H^2,$ $p\neq \alpha(-\infty), \alpha(+\infty).$ For some small $\epsilon>0,$ consider the asymptotic vertical segment joining $(p,\epsilon)$ and $(p,h+\epsilon).$ Now, connect $(p,\epsilon)$ to $(\alpha(-\infty),0)$ and $(p,h+\epsilon)$ to $(\alpha(-\infty),2h)$ by curves in $\partial_\infty \hr$, whose tangent vectors are never horizontal or vertical, and so that the resulting curve $\Gamma$ is differentiable. Also, consider the horizontal geodesic $\beta$ connecting $p$ to $\alpha(+\infty).$

Parametrize $\alpha$ by arc length, and consider $\gamma$ a geodesic orthogonal to $\alpha$ passing through $\alpha(0).$ Let us denote by $d(t)$ the equidistant curve to $\gamma$ at a distance $|t|$ that intersects $\alpha$ at $\alpha(t).$ For each $t$ consider a curve $\Gamma_t$ contained in the plane $d(t)\times\rr$ with endpoints $(\alpha(t),0)$ and $(\alpha(t),2h)$ such that $\Gamma_t$ is contained in the region $R$ bounded by $\alpha\times\rr,\beta\times\rr, \hh\times\{0\}$ and $\hh\times\{2h\}$ with the properties that its tangent vectors do not point in the horizontal direction and $\Gamma_t$ converges to $\Gamma$ when $t\rightarrow -\infty.$ In particular, $\Gamma_t$ is transversal to the Killing field $Y$ whose flow $(\phi_l)_{l\in\mathbb R}$ is given by translation along the geodesic $\gamma.$

Write $\alpha_t$ to denote the vertical segment $\alpha(t)\times[0,2h]$ (see Figure \ref{exem21}).

\begin{figure}[h!]
 \centering
\includegraphics[width=6.5cm]{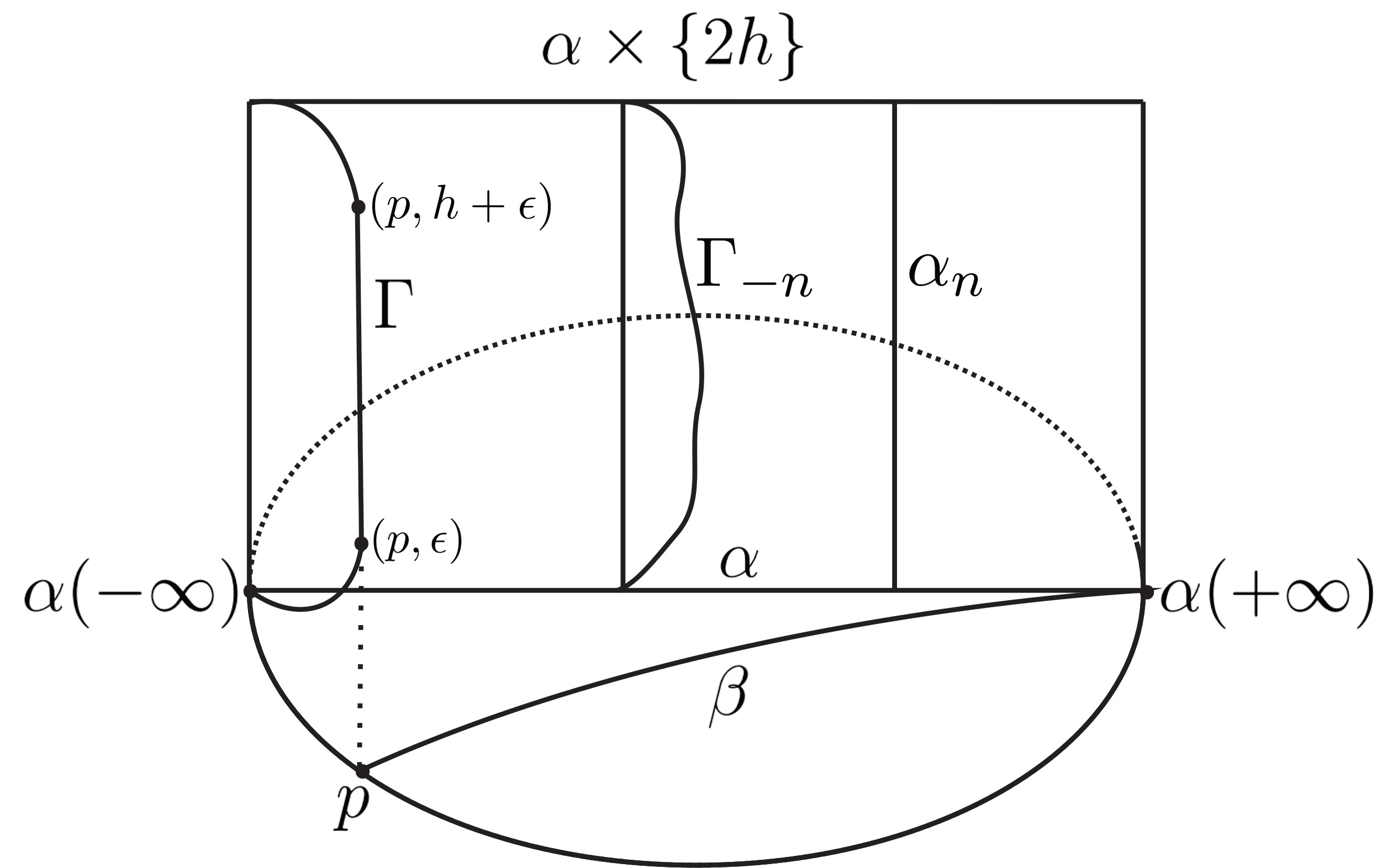}
\caption{Curves $\Gamma_{-n}$ and $\Gamma.$}
\label{exem21}
\end{figure}

For each $n,$ let $\Sigma_n$ be the solution to the Plateau problem with boundary $$\Gamma_{-n}\cup(\alpha(\left[-n,n\right])\times\{0\})\cup(\alpha(\left[-n,n\right])\times\{2h\})\cup\alpha_{n}.$$ By our choice of the curves $\Gamma_t,$ the boundary $\partial\Sigma_n$ is transverse to the Killing field $Y.$ Using the foliation of $\hr$ by the vertical planes $\phi_l(\alpha), l\in\mathbb R,$ the Alexandrov reflection technique shows that $\Sigma_n$ is a $Y$-Killing graph. In particular, it is unique and stable \cite{NR-Simply}, and we have uniform curvature estimates far from the boundary \cite{RosenbergSouam}. When we apply the rotation by angle $\pi$ around $\alpha\times\{0\}$ to the minimal surface $\Sigma_n,$ we get another minimal surface which extends $\Sigma_n,$ is still a Y-Killing graph and has int$(\alpha(\left[-n,n\right])\times\{0\})$ in its interior. Hence we obtain uniform curvature estimates for $\Sigma_n$ in a neighborhood of $\alpha(\left[-n,n\right])\times\{0\}.$ This is also true for $\alpha(\left[-n,n\right])\times\{2h\}$ and $\alpha_{n}.$

Observe that $\Sigma_n$ is contained in the region $R,$ for all $n.$

By our choice of $\Gamma,$ for each $q\in\Gamma,$ we can consider two translations of the minimal surfaces $S_h$ (considered in the last section) that pass through $q$ so that one of them has asymptotic boundary under $\Gamma,$ the other one has asymptotic boundary above $\Gamma$ and their intersection with $\Gamma$ is just the point $q$ considered or is the whole vertical segment $\overline{(p,\epsilon)(p,h+\epsilon)}$. Hence, the envelope of the union of all these translated surfaces $S_h$ forms a barrier to $\Sigma_n,$ for all $n.$ 

Then, as we have uniform curvature estimates and barriers at infinity, we conclude that there exists a subsequence of $\Sigma_n$ that converges to a minimal surface $\Sigma$ with $(\alpha(+\infty)\times[0,2h])\cup\Gamma=\partial_\infty\Sigma,$ and then $$\partial\Sigma=\Gamma\cup(\alpha\times\{0\})\cup(\alpha\times\{2h\})\cup(\alpha(+\infty)\times[0,2h]).$$

Therefore, the surface obtained by reflection in all horizontal boundary geodesics of $\Sigma$ is a vertically periodic minimal surface invariant by $T(4h).$ In the quotient space this minimal surface has two ends; one is asymptotic to a vertical plane and has finite total curvature, while the other one is topologically an annular end and has infinite total curvature. 

\begin{proposition}
There exists a singly periodic minimal surface (invariant by a vertical translation) such that, in the quotient space, this surface has two ends; one end is asymptotic to a vertical plane and has finite total curvature, while the other one is topologically an annular end and has infinite total curvature.
\end{proposition}

\subsection{Periodic minimal surfaces invariant by screw motion}

Now we construct some examples of periodic minimal surfaces invariant by a screw motion, that is, invariant by a subgroup of isometries generated by the composition of a horizontal translation with a vertical translation. 

Consider two geodesics $\alpha, \beta$ in $\mathbb H^2$ that limit to the same point at infinity, say $\alpha(+\infty)=p_0=\beta(+\infty)$. For $h>\pi,$ consider a smooth curve $\Gamma$ contained in the asymptotic boundary of $\hr,$ connecting $(\alpha(-\infty),2h)$ to $(\beta(-\infty),0)$ and such that its tangent vectors are never horizontal or vertical. Also, take a point $p\in\partial_\infty\mathbb H^2$ in the halfspace determined by $\beta\times\rr$ that does not contain $\alpha.$

For some small $\epsilon>0,$ consider the asymptotic vertical segment joining $(p,\epsilon)$ and $(p,h+\epsilon).$ Now, connect $(p,\epsilon)$ to $(p_0,0)$ and $(p,h+\epsilon)$ to $(p_0,2h)$ by curves in $\partial_\infty \hr$ whose tangent vectors are never horizontal or vertical, and such that the resulting curve $\widehat{\Gamma}$ is differentiable.

Parametrize $\alpha$ by arc length, and consider $\gamma$ a geodesic orthogonal to $\alpha$ passing through $\alpha(0).$ Let us denote by $d(t)$ the equidistant curve to $\gamma$ at a distance $|t|$ that intersects $\alpha$ at $\alpha(t).$ For each $t,s$ consider two curves $\widehat{\Gamma_t}$ and $\Gamma_s$ contained in the plane $d(t)\times\rr$ and $d(s)\times\rr,$ respectively, with the properties that their tangent vectors are never horizontal, $\widehat{\Gamma_t}$ joins $(\alpha(t),2h)$ to $(\beta(t),0),$ $\Gamma_s$ joins $(\alpha(s),2h)$ to $(\beta(s),0),$ $\widehat{\Gamma_t}$ converges to $\widehat{\Gamma}$ when $t\rightarrow+\infty,$ $\Gamma_s$ converges to $\Gamma$ when $s\rightarrow-\infty,$ and both curves are contained in the region $R$ bounded by $\alpha\times\rr,\theta\times\rr, \hh\times\{0\}$ and $\hh\times\{2h\},$ where $\theta$ is the geodesic with endpoints $p$ and $\beta(-\infty)$ (see Figure \ref{exem31}).

\begin{figure}[h]
 \centering
\includegraphics[width=6cm]{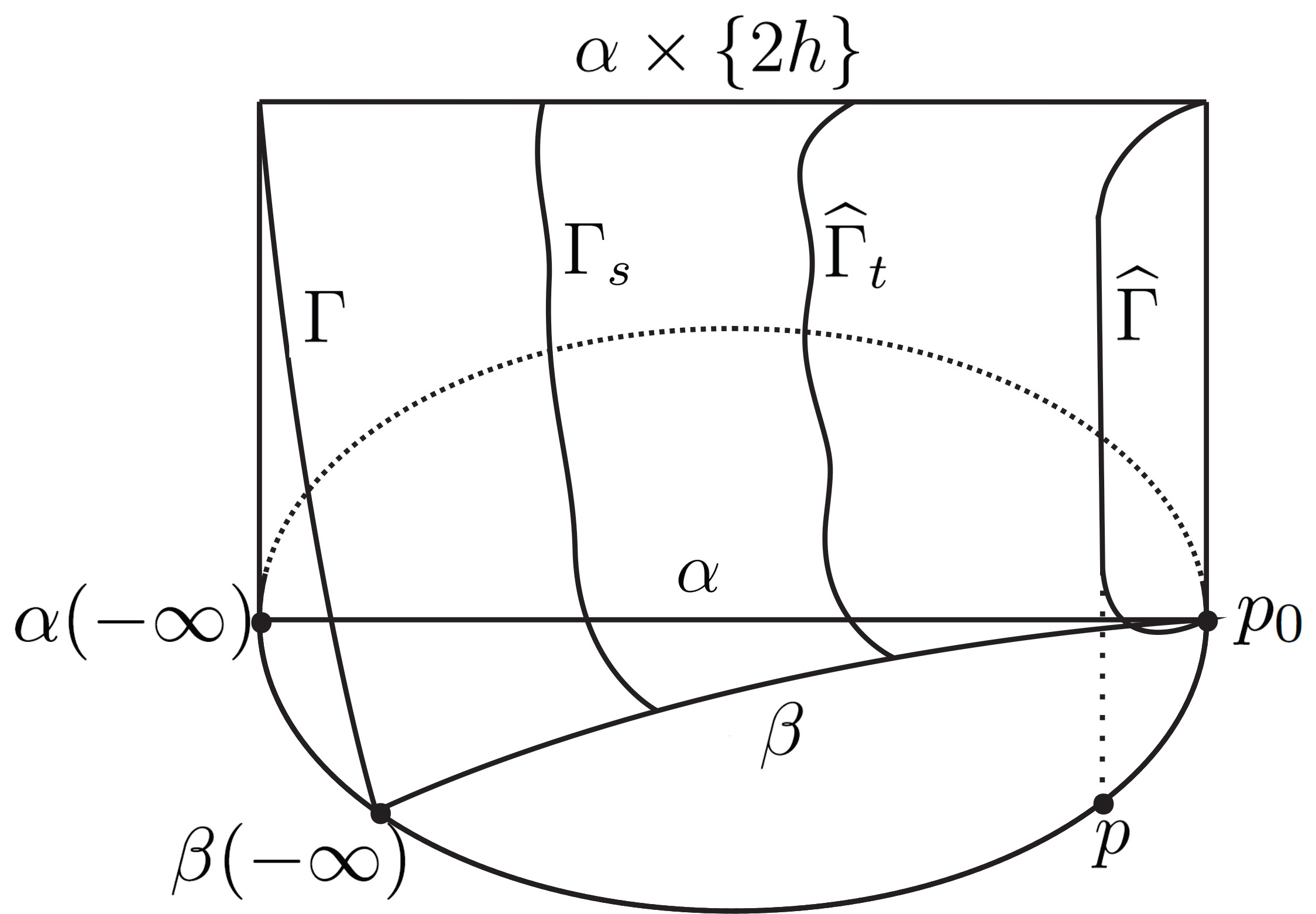}
\caption{Curves $\widehat{\Gamma}_t, \Gamma_s, \widehat{\Gamma}$ and $\Gamma.$}
\label{exem31}
\end{figure}

For each $n,$ let $\Sigma_n$ be the solution to the Plateau problem with boundary $$\Gamma_{-n}\cup(\alpha([-n,n])\times\{2h\})\cup\widehat{\Gamma}_n\cup(\beta([-n,n])\times\{0\}).$$ The surface $\Sigma_n$ is contained in the region $R.$ As in the previous section, we can show that $\Sigma_n$ is a Killing graph, then it is stable, unique and we have uniform curvature estimates far from the boundary. Rotating $\Sigma_n$ by angle $\pi$ around the geodesic $\alpha\times\{2h\}$ we get a minimal surface which extends $\Sigma_n,$ is still a Killing graph, and has int$(\alpha([-n,n])\times\{2h\})$ in its interior. Hence we get uniform curvature estimates for $\Sigma_n$ in a neighborhood of $\alpha([-n,n])\times\{2h\}.$ This is also true for $\beta([-n,n])\times\{0\}.$ Thus when $n\rightarrow+\infty$, there exists a subsequence of $\Sigma_n$ that converges to a minimal surface $\Sigma$ with $\Gamma\cup\widehat{\Gamma}\subset\partial_\infty\Sigma_n.$ Using the same argument as before with suitable translations of the surface $S_h$ as barriers, we conclude that in fact $\partial_\infty\Sigma=\Gamma\cup\widehat{\Gamma},$ and then $\partial\Sigma=\Gamma\cup(\alpha\times\{2h\})\cup(\beta\times\{0\})\cup\widehat{\Gamma}.$

The surface obtained by reflection in all horizontal boundary geodesics of $\Sigma$ is a minimal surface invariant by $\psi^2\circ T(4h),$ where $\psi$ is the horizontal translation along horocycles that sends $\alpha$ to $\beta.$ There are two annular embedded ends in the quotient, each of infinite total curvature. 

\begin{proposition}
There exists a minimal surface invariant by a screw motion such that, in the quotient space, this minimal surface has two annular embedded ends, each one of infinite total curvature.
\end{proposition}

Now we will construct another interesting example of a periodic minimal surface invariant by a screw motion.

Denote by $\gamma_0, \gamma_1$ the geodesic lines $\{x=0\}, \{y=0\}$ in $\mathbb H^2,$ respectively. Let $c$ be a horocycle orthogonal to $\gamma_1,$ and consider $p,q\in c$ equidistant points to $\gamma_1.$ Take $\alpha, \beta$ geodesics which limit to $p_0=(1,0)=\gamma_1(+\infty)$ and pass through $p,q,$ respectively. Fix $\epsilon>0$ and $h>\pi.$ Denote the points $A=\alpha(-t_0),C=\alpha(t_0),B=\beta(-t_0), D=\beta(t_0),$ and let us consider the following Jordan curve (see Figure \ref{gammat0}):
$$
\begin{array}{rcl}
\Gamma_{t_0}&=&(\alpha([-t_0,t_0])\times\{-\epsilon\})\cup\overline{(C,-\epsilon)(D,0)}\cup(\beta ([-t_0,t_0])\times\{0\})\\
&&\\
&&\cup(\alpha([-t_0,t_0])\times\{h\})\cup\overline{(C, h)(D, h+\epsilon)}\cup(\beta([-t_0,t_0])\times\{h+\epsilon\})\\
&&\\
&&\cup\overline{(A, -\epsilon)(A, h)}\cup\overline{(B, 0)(B, h+\epsilon)}.
\end{array}
$$

\begin{figure}[h]
\centering
\includegraphics[height=5cm]{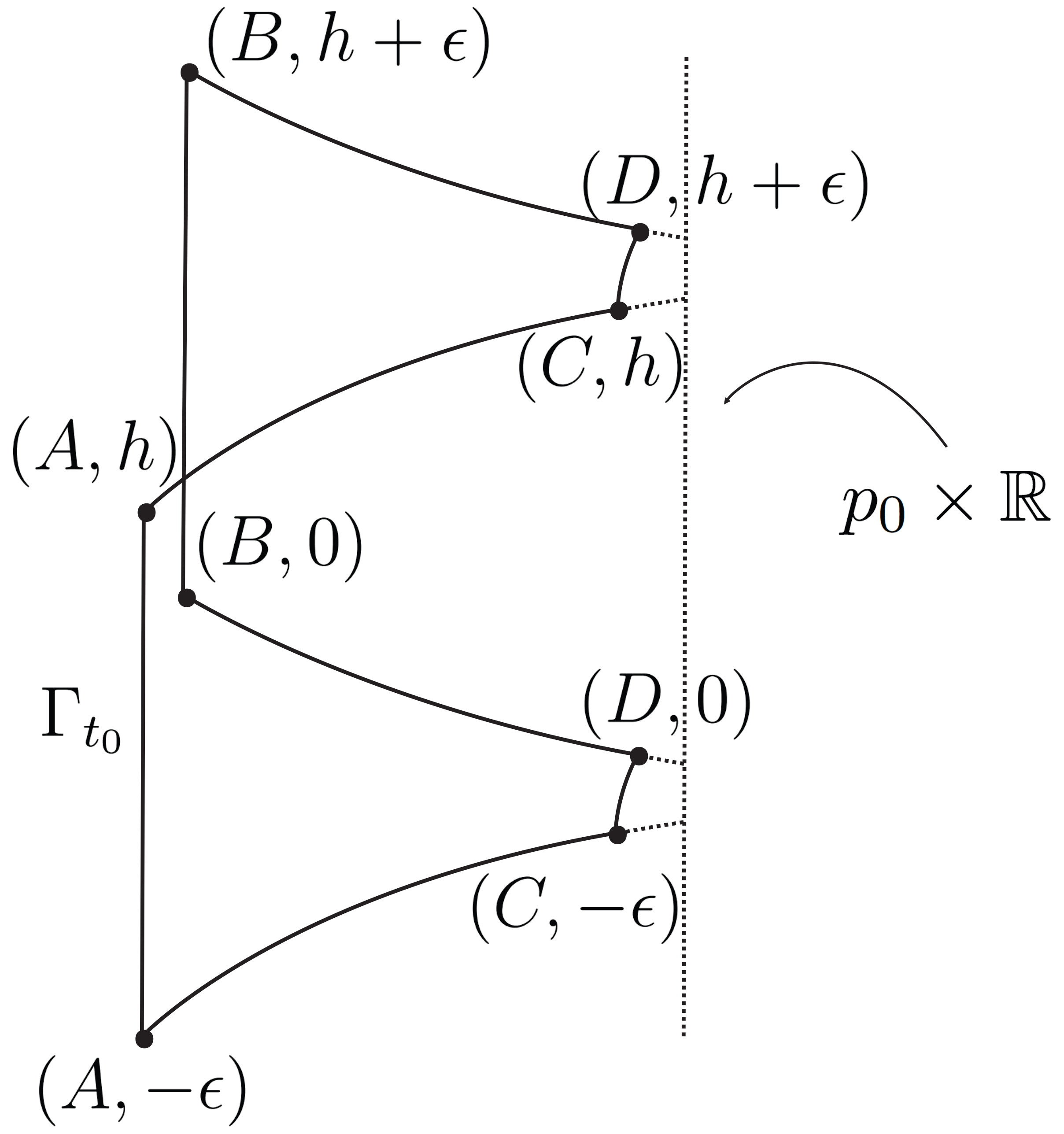}
\caption{Curve $\Gamma_{t_0}$.}
\label{gammat0}
\end{figure}

We consider a least area embedded minimal disk $\Sigma_{t_0}$ with boundary $\Gamma_{t_0}.$ 

Denote by $Y_1$ the Killing vector field whose flow $(\phi_l)_{l\in (-1,1)}$ gives the hyperbolic translation along $\gamma_1$ with $\phi_l(0)=$ $(l, 0)$ and $p_0$ as attractive point at infinity. As $\Gamma_{t_0}$ is transversal to the Killing field $Y_1,$ we can prove, using the Alexandrov reflection procedure, that $\Sigma_{t_0}$ is a $Y_1$-Killing graph with convex boundary, in particular, $\Sigma_{t_0}$ is stable and unique \cite{NR-Simply}. This yields uniform curvature estimates far from the boundary \cite{RosenbergSouam}. Rotating $\Sigma_{t_0}$ by angle $\pi$ around the geodesic arc $\alpha([-t_0,t_0])\times\{-\epsilon\}$ gives a minimal surface that extends $\Sigma_{t_0},$ has int$(\alpha([-t_0,t_0])\times\{-\epsilon\})$ in its interior, and is still a $Y_1$-Killing graph. Thus we get uniform curvature estimates for $\Sigma_{t_0}$ in a neighborhood of $\alpha([-t_0,t_0])\times\{-\epsilon\}.$ This is also true for the three other horizontal geodesic arcs in $\Gamma_{t_0}.$

Write $F=\alpha(-\infty), G=\beta(-\infty).$ Observe that, by the maximum principle, for any $t_0,$ $\Sigma_{t_0}$ stays in the halfspace determined by $\overline{FG}\times\rr$ that contains $\Gamma_{t_0}.$

Since $h>\pi,$ we can consider the minimal surface $S_h$ (considered in Section \ref{sec-41}) as a barrier. For $l$ close to $1,$ the translated surface $\phi_l(S_h)$ does not intersect $\Sigma_{t_0}.$ 
 The surface $\Sigma_{t_0}$ is contained between $\phi_{l}(S_h)$ and $\overline{FG}\times\rr.$ 
 When $t_0\rightarrow +\infty,$ $\Gamma_{t_0}$ converges to $\Gamma,$ where
$$\begin{array}{rcl}
\Gamma&=&(\alpha\times\{-\epsilon\})\cup\overline{(p_0,-\epsilon)(p_0,0)}\cup(\beta\times\{0\})\\
&&\\
&&\cup(\alpha\times\{h\})\cup\overline{(p_0, h)(p_0, h+\epsilon)}\cup(\beta\times\{h+\epsilon\})\\
&&\\
&&\cup\overline{(F, -\epsilon)(F, h)}\cup\overline{(G, 0)(G, h+\epsilon)}.
\end{array}$$
 
Using the maximum principle, we can prove that $\Sigma_t$ is contained between $\phi_{l}(S_h)$ and $\overline{FG}\times\rr,$ for all $t>t_0.$ Therefore, there exists a subsequence of the surfaces $\Sigma_{t}$ that converges to a minimal surface $\Sigma,$ where $\Sigma$ lies in the region between $\hh\times\{-\epsilon\}$ and $\hh\times\{h+\epsilon\}$ bounded by $\alpha\times \rr,$ $\beta\times\rr,$ $\overline{FG}\times\rr$ and $\phi_{l}(S_h),$ and has boundary $\partial\Sigma=\Gamma.$

Hence the surface obtained by reflection in all horizontal boundary geodesics of $\Sigma$ is invariant by $\psi^2\circ T(2(h+\epsilon)),$ where $\psi$ is the horizontal translation along horocycles that sends $\alpha$ to $\beta.$ Moreover, this surface in the quotient space has four ends: two vertical ends and two helicoidal ends, all of them with finite total curvature.

\begin{proposition}
There exists a minimal surface invariant by a screw motion such that, in the quotient space, this minimal surface has four ends. Two vertical ends and two helicoidal ends, all of them with finite total curvature.
\end{proposition}

\section{A multi-valued Rado Theorem}
\label{sec-rado}

The aim of this section is to prove a multi-valued Rado theorem for small perturbations of the helicoid. Recall that Rado's theorem says that minimal surfaces over a convex domain with graphical boundaries must be disks which are themselves graphical. We will prove that for certain small perturbations of the boundary of a (compact) helicoid there exists only one compact minimal disk with that boundary. By a compact helicoid we mean the intersection of a helicoid with certain compact regions  in $\hr.$ The idea here originated in the work of Hardt and Rosenberg \cite{HardtRosenberg}. We will apply this multi-valued Rado theorem to construct an embedded minimal surface in $\hr$ whose boundary is a small perturbation of the boundary of a complete helicoid. 

Consider $Y$ the Killing field whose flow $(\phi_\theta)_{\theta\in[0,2\pi)}$ is given by rotations around the $z$-axis. For some $0<c<1,$ let $D=\{(x,y)\in \mathbb H^2; x^2+y^2\leq c\}.$ Take a helix $h$ of constant pitch contained in a solid cylinder $D\times [0,d]$, so that the vertical projection of $h$ over $\hh\times\{0\}$ is $\partial D.$ Let us denote by $\Gamma$ the Jordan curve which is the union of $h$, the two horizontal geodesic arcs joining the endpoints of $h$ to the $z$-axis, and the part of the $z$-axis. Call $\mathcal H$ the compact part of the helicoid that has $\Gamma$ as its boundary. We know that $\mathcal H$ is a minimal surface transversal to the Killing field $Y$ at the interior points. Take $\theta<\pi/4,$ and consider $\mathcal H_1=\phi_{-\theta}(\mathcal H)$ and $\mathcal H_2=\phi_{\theta}(\mathcal H).$ Hence $\mathcal H_1, \mathcal H_2$ are two compact helicoids with boundary $\partial\mathcal H_1=\phi_{-\theta}(\Gamma),$ $\partial\mathcal H_2=\phi_{\theta}(\Gamma).$

Consider $h_0$ a small smooth perturbation of the helix $h$ with fixed endpoints such that $h_0$ is transversal to $Y$ and $h_0$ is contained in the region between $\phi_{-\theta}(h)$ and $\phi_{\theta}(h)$ in $\partial D\times [0,d]$. Call $\Gamma_0$ the Jordan curve which is the union of $h_0,$ the two horizontal geodesic arcs and a part of the $z$-axis, hence $\Gamma_0=(\Gamma\setminus h)\cup h_0$ (see Figure \ref{gamma}).

\begin{figure}[h]
 \centering
\includegraphics[height=3.5cm]{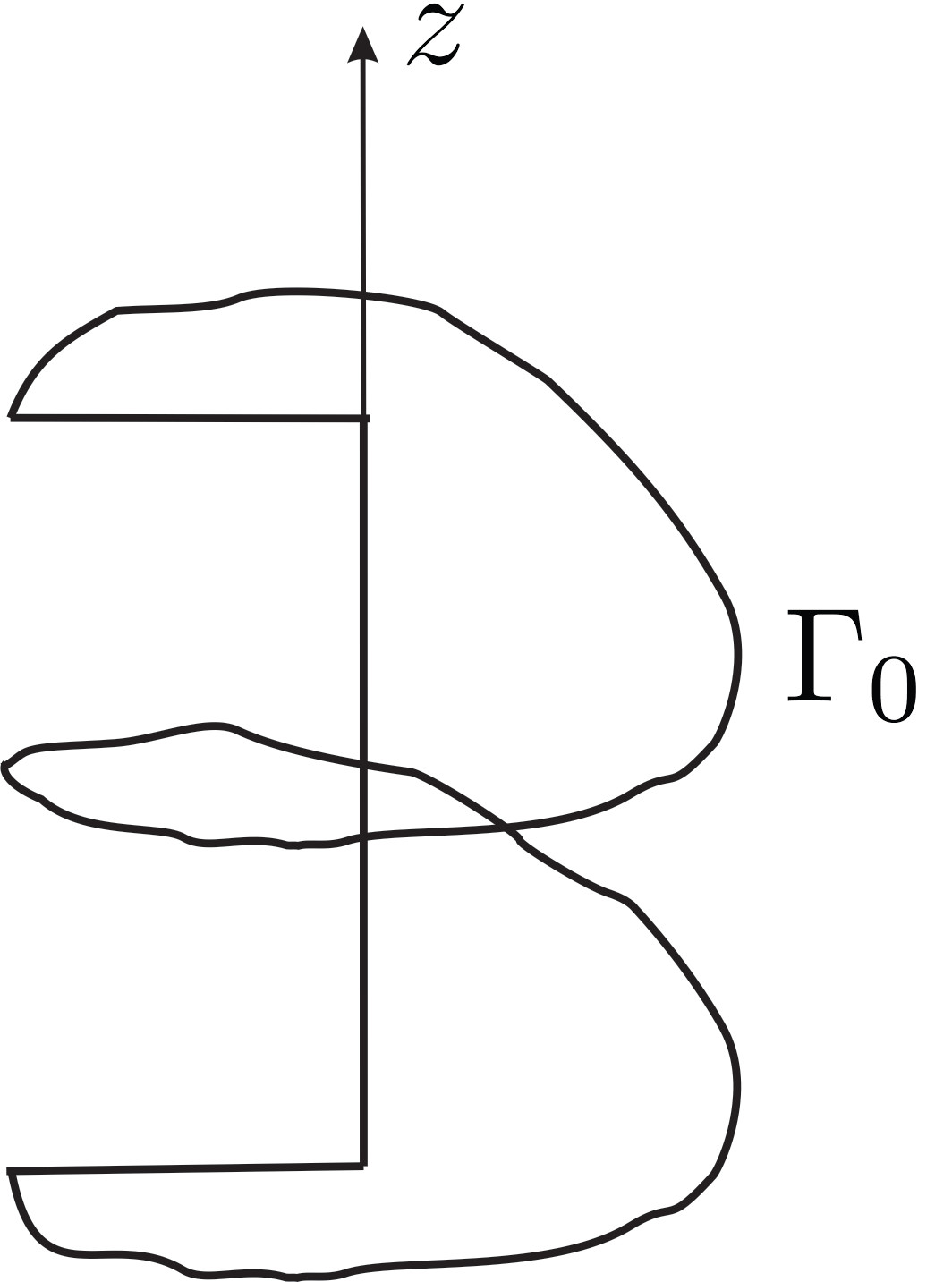}
\caption{Curve $\Gamma_0.$}
\label{gamma}
\end{figure}

Denote by $R$ the convex region bounded by $\mathcal H_1$ and $\mathcal H_2$ in the solid cylinder $D\times [0,d].$ The Jordan curve $\Gamma_0$ is contained in the simply connected region $R$ which has mean convex boundary. Then we can consider the solution to the Plateau problem in this region $R,$ and we get a compact minimal disk $H$ contained in $R$ with boundary $\partial H=\Gamma_0.$

\begin{proposition}
Under the assumptions above, $H$ is transversal to the Killing field $Y$ at the interior points. Moreover, the family $(\phi_{\theta}(H))_{\theta\in[0,2\pi)}$ foliates $D\times [0,d]\setminus \{z\mbox{-axis}\}.$
\label{transversal}
\end{proposition}

\begin{proof}
As $H$ is a disk, we already know that each integral curve of $Y$ intersects $H$ in at least one point. 

Observe that $\phi_{\pi/2}(R)\cap R\setminus \{z\mbox{-axis}\}=\emptyset$ and, in particular, $\phi_{\pi/2}(H)\cap H\setminus \{z\mbox{-axis}\}=\emptyset.$ Moreover, notice that the tangent plane of $\phi_{\pi/2}(H)$ never coincides with the tangent plane of $H$ along the $z$-axis; at each point of the $z$-axis the surfaces are in disjoint sectors. So as one decreases $t$ from $\pi/2$ to $0,$ the surfaces $\phi_{t}(H)$ and $H$ have only the $z$-axis in common and they are never tangent along the $z$-axis.  More precisely, as $t$ decreases, $t > 0$, there can not be a first interior point of contact between the two surfaces by the maximum principle.  Also there can not be a point on the $z$-axis which is a first point of tangency of the two surfaces for $t > 0$, by the boundary maximum principle.  Thus the surfaces $\phi_t(H)$ and $H$ have only the $z$-axis in common for $0< t \leq \pi/2$. The same argument works for $-\pi/2 \leq t< 0$. Thus each integral curve of $Y$ intersects $H$ in exactly one point. 

Denote by $R_2$ the region in $R$ bounded by $H$ and $\mathcal H_2,$ and denote by $N$ the unit normal vector field of $H$ pointing toward $R_2.$ As each integral curve of $Y$ intersects $H$ in exactly one point, we have $\left\langle N, Y\right\rangle\geq 0$ on $H.$ As $\left\langle N, Y\right\rangle$ is a Jacobi function on the minimal surface $H,$ we conclude that necessarily $\left\langle N, Y\right\rangle> 0$ in int$H$. Therefore,  $H$ is transversal to the Killing field $Y$ at the interior points, and the surfaces $\phi_t(H)$ foliate $D\times[0,d]\setminus \{z\mbox{-axis}\}$ for $t\in[0,2\pi)$.
\end{proof}

\begin{theorem}[A multi-valued Rado Theorem]
Under the assumptions above, $H$ is the unique compact minimal disk with boundary $\Gamma_0.$
\label{rado}
\end{theorem}

\begin{proof}
Set $\Gamma_\theta=\phi_\theta (\Gamma_0)$ and $H_\theta=\phi_\theta(H),$ so $H_\theta$ is a minimal disk with $\partial H_\theta=\Gamma_\theta.$ By Proposition \ref{transversal}, the family $(H_\theta)_{\theta\in[0,2\pi)}$ gives a foliation of the region $D\times [0,d]\setminus \{z\mbox{-axis}\}.$

Let $M\neq H$ be another compact minimal disk with boundary $\Gamma_0.$ We will analyse the intersection between $M$ and each $H_\theta.$

First, observe that $M\cap H_\theta \neq\emptyset$ for all $\theta$ and by the maximum principle $M\subset D\times [0,d].$

Fix $\theta_0.$ Given $q\in H_{\theta_0}\cap M,$ then either $q\in \mbox{int}M$ or $q\in \Gamma_0=\partial M.$ 

Suppose $q\in\mbox{int}M.$

If the intersection is transversal at $q,$ then in a neighborhood of $q$ we have that $H_{\theta_0}\cap M$ is a simple curve passing through $q.$ If we let $\theta_0$ vary a little, we see in $M$ a foliation as in part $(a)$ of Figure \ref{fig8}. 

\begin{figure}[h]
 \centering
\includegraphics[height=4.7cm]{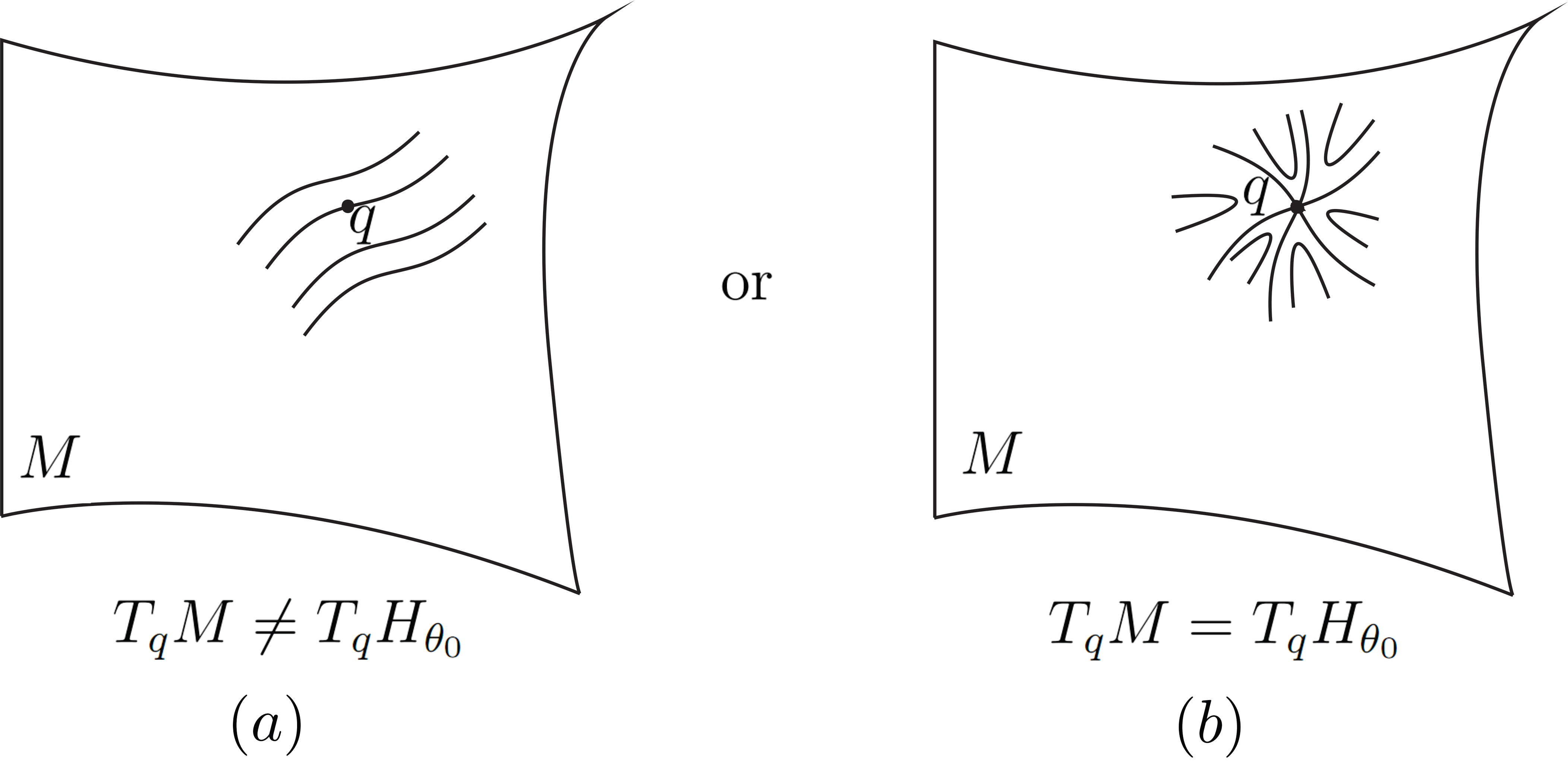}
\caption{$q\in\mbox{int}M.$}
\label{fig8}
\end{figure}

On the other hand, if $M$ is tangent to $H_{\theta_0}$ at $q,$ as the intersection of any two minimal surfaces is locally given by an $n$-prong singularity, that is, $2n$ embedded arcs which meet at equal angles (see \cite{HoffmanMeeks}, Claim $1$ of Lemma $4$), then in a neighborhood of $q$ we have that $H_{\theta_0}\cap M$ consists of $2n$ curves  passing through $q$ and making equal angles at $q$. If we let $\theta_0$ vary a little, we see in $M$ a foliation as in part $(b)$ of Figure \ref{fig8}.

Now suppose $q\in \Gamma_0.$ 

If $q\in\Gamma_0\cap\{z\mbox{-axis}\},$ to understand the trace of $H_{\theta_0}$ on $M$ in a neighborhood of $q$ we proceed as follows. Rotation by angle $\pi$ of $\hr$ about the $z$-axis extends $M$ smoothly to a minimal surface $\widetilde{M}$ that has $q$ as an interior point. Each $H_\theta$ also extends by this rotation (giving a helicoid $\widetilde{H}_\theta$). So in a neighborhood of $q,$ we understand the intersection of $\widetilde{M}$ and $\widetilde{H}_{\theta_0}.$ The surfaces $\widetilde{\mathcal M}$ and $\widetilde{H}_{\theta_0}$ are either transverse or tangent at $q$ as in Figure \ref{fig8}. Then when we restrict to $M\cap H_{\theta_0}$ and let $\theta_0$ vary slightly, we see that the trace of $H_{\theta_0}$ on $M$ near $q$ is as in Figure \ref{figure22}, since the segment on the $z$-axis through $q$ is in $M\cap H_{\theta_0}.$


\begin{figure}[h]
 \centering
\includegraphics[height=3.4cm]{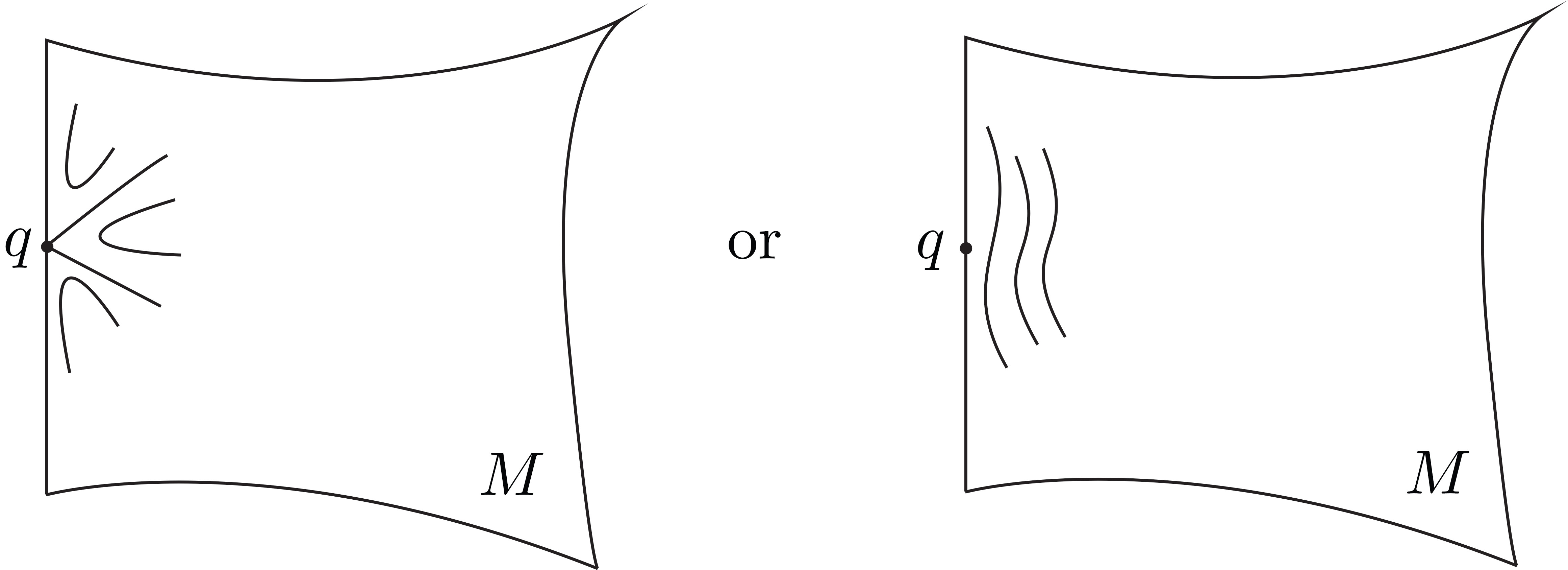}
\caption{$q\in\Gamma_0\cap\{z\mbox{-axis}\}.$}
\label{figure22}
\end{figure}

On the other hand, if $q\in\Gamma_0\setminus\{z\mbox{-axis}\}$ then $\theta_0=0,$ since $\Gamma_\theta\cap\Gamma_0\setminus\{z\mbox{-axis}\}=\emptyset$ for any $\theta\neq 0.$ Note that we cannot have $M\cap H$ homeomorphic to a semicircle in a neighborhood of $q,$ since this would imply that $M$ is on one side of $H$ at $q$ and this contradicts the boundary maximum principle. Thus when we let $\theta_0=0$ vary a little, we have two possible foliations for $M$ in a neighborhood of $q$ as indicated in Figure \ref{fig11}.

\begin{figure}[h]
 \centering
\includegraphics[height=2cm, width=8cm]{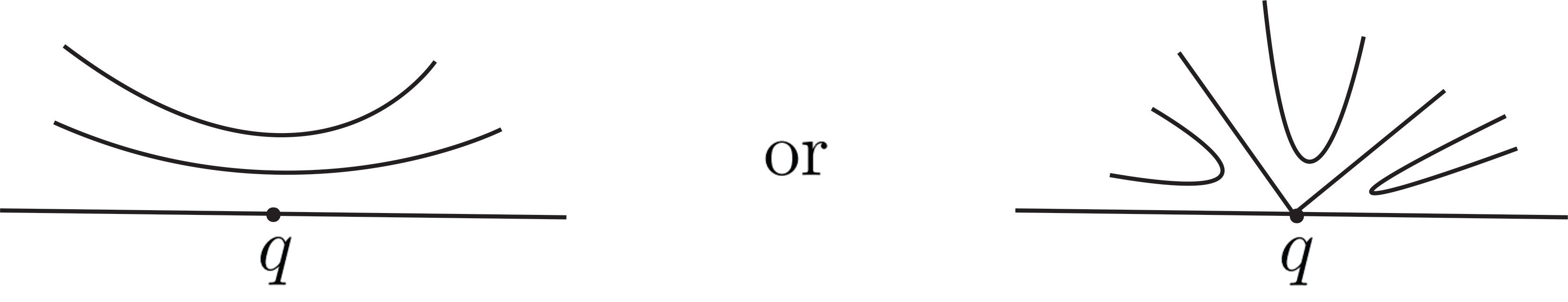}
\caption{$q\in\Gamma_0\setminus\{z\mbox{-axis}\}.$}
\label{fig11}
\end{figure}

Now consider two copies of $M$ and glue them together along the boundary.


Since $M$ is a disk, when we glue these two copies of $M$ we obtain a sphere with a foliation whose singularities have negative index by the analysis above. But this is impossible. Therefore, there is no other minimal disk with boundary $\Gamma_0$ besides $H.$

\end{proof}
 
\begin{remark}
This proof clearly works to prove Theorem 2 for slightly perturbed helicoids in $\mathbb R^3.$
\end{remark}

Now let us construct an example of a complete embedded minimal surface in $\hr$ whose asymptotic boundary is a small perturbation of the asymptotic boundary of a complete helicoid.

Consider the (compact) helix $\beta(u)=(\cos u, \sin u, 2u)$ for $u\in[0, 4\pi].$ Notice that $\beta$ is a multi-graph over $\partial_\infty\hh.$ Take $\theta< \pi/4$ and consider $\alpha(u)$ a small perturbation of $\beta(u)$ in $\partial_\infty\hr$ contained between $\phi_{-\theta}(\beta)$ and $\phi_{\theta}(\beta)$ such that $\alpha$ is transversal to $\partial_t$ and $\partial_\infty\hh \times \{\tau\}$ for any $\tau\in[0, 8 \pi],$ $\alpha(0)=\beta(0),$ $\alpha(4\pi)=\beta(4\pi)$ and so that the vertical distance between $\alpha(s)$ and $\alpha(s+2\pi)$ is bigger than $\pi$ for any $s\in(0,2\pi).$

Now for $t\in[0, 1],$ consider the curves $\alpha_t(u)=(1-t)(0,0,u)+t\alpha(u), u\in[0,4\pi].$  Call $\Gamma_t$ (respectively $\Gamma_1$) the Jordan curve which is the union of $\alpha_t$ (respectively $\alpha$), the two horizontal geodesics joining the endpoints of $\alpha_t$ (respectively $\alpha$)  to the $z$-axis, and the part of the $z$-axis between $z=0$ and $z=8\pi$. Note that when $t$ goes to $1,$ the curves $\Gamma_t$ converge to the curve $\Gamma_1.$ Denote by $H_t$ the minimal disk with boundary $\Gamma_t.$ By Theorem \ref{rado}, $H_t$ is stable and unique. In particular, we have uniform curvature estimates for points far from the boundary. As before, using rotation by angle $\pi$ around horizontal geodesics, we can prove that there is uniform curvature estimates for $H_t$ in a neighborhood of the two horizontal geodesic arcs of $\Gamma_t.$ 

As in the previous section, the envelope of the union of the translated surfaces $S_\pi$ forms a barrier to the sequence $H_t,$ hence we conclude that there exists a subsequence of $H_t$ that converges to a minimal surface $H_1$ with boundary $\partial H_1=\Gamma_1.$ Rotation by angle $\pi$ of $\hr$ around the $z$-axis extends $H_1$ smoothly to a minimal surface which has two horizontal (straight) geodesics in its boundary. Thus the surface obtained by reflection in all horizontal boundary geodesics of $H_1$ is a minimal surface whose asymptotic boundary is a small perturbation of the asymptotic boundary of the complete helicoid in $\hr$ which has $\beta$ contained in its asymptotic boundary.
 

\bibliographystyle{line}
\bibliography{alexandrov}

\begin{flushleft}
\textsc{Instituto Nacional de Matem\' atica Pura e Aplicada (IMPA)}

\textsc{Estrada Dona Castorina 110, 22460-320, Rio de Janeiro-RJ, Brazil}

\textit{Email adress:} anamaria@impa.br
\end{flushleft}

\end{document}